\documentclass[10pt]{amsart}
  \baselineskip 1 mm

\usepackage{amsmath,amsthm,amsfonts,latexsym,amsopn,verbatim,amscd,amssymb}
\usepackage{hyperref}

\theoremstyle{plain}
\newtheorem{theorem}{Theorem}[section]

\theoremstyle{definition}

\numberwithin{equation}{section}

\DeclareMathOperator{\spec}{Spec}
\DeclareMathOperator{\L-spec}{L-spec}
\DeclareMathOperator{\Q-spec}{Q-spec}

\newcommand{\bnum}{\begin{enumerate}}
\newcommand{\enum}{\end{enumerate}}

\begin{document}

\title[Spectrum and energies of commuting conjugacy class graph of a finite  group]{Spectrum and energies of commuting conjugacy class graph of a finite  group}
\author[P. Bhowal and R. K. Nath]{Parthajit Bhowal  and Rajat Kanti Nath$^*$}
\address{Department of Mathematical Sciences, Tezpur
University,  Napaam-784028, Sonitpur, Assam, India.
}
\email{bhowal.parthajit8@gmail.com, rajatkantinath@yahoo.com (correcponding author)}

\subjclass[2010]{20D99, 05C50, 15A18, 05C25.}
\keywords{commuting conjugacy class graph, spectrum, energy, finite group.}

%

\thanks{*Corresponding author}

%
%
\begin{abstract}
In this paper we compute   spectrum, Laplacian spectrum, signless Laplacian spectrum  and their corresponding energies of commuting conjugacy class graph of the group $G(p, m, n) = \langle x, y : x^{p^m} = y^{p^n} = [x, y]^p = 1, [x, [x, y]] = [y, [x, y]] = 1\rangle$, where $p$ is any prime,  $m \geq 1$ and $n \geq 1$. We derive some consequences along with the fact that commuting conjugacy class graph of  $G(p, m, n)$ is super integral.
We also compare various energies and determine whether commuting conjugacy class graph of  $G(p, m, n)$  is hyperenergetic, L-hyperenergetic or Q-hyperenergetic. 
\end{abstract}

\maketitle

%
%

\section{Introduction} \label{S:intro}

Let $G$ be any group and $V(G) = \{x^G : x \in G \setminus Z(G)\}$, where $x^G$ is the conjugacy class of $x$ in $G$ and $Z(G)$ is the center of $G$. We consider the graph $\mathcal{CCC}(G)$, called commuting conjugacy class graph of $G$, with vertex set $V(G)$ and  two distinct vertices $x^G$ and $y^G$ are adjacent if there exists some elements $x' \in x^G$ and $y' \in y^G$ such that $x'$ and $y'$ commute. Extending the notion of commuting graph of a group pioneered by Brauer and Fowler \cite{bF1955}, Herzog et al.  \cite{hLM2009}  were introduced commuting conjugacy class graph of groups   in the first decade of this millennium. The second and third paper on this topic got published in the years 2016 and 2020 authered by  Mohammadian et al. \cite{mefw} and  Salahshour et al. \cite{sA2020} respectively, where  Mohammadian et al. characterize finite groups such that $\mathcal{CCC}(G)$ is triangle-free and Salahshour et al. describe the structure of  
 $\mathcal{CCC}(G)$ considering $G$ to be dihedral groups $(D_{2n} \text{ for }  n \geq 3)$, generalized  quaternion groups $(Q_{4m} \text{ for } m \geq 2)$, semidihedral groups $(SD_{8n} \text{ for } n \geq 2)$, the groups $V_{8n} = \langle x, y : x^{2n} = y^4 = 1, yx = x^{-1}y^{-1}, y^{-1}x = x^{-1}y \rangle$   (for   $n \geq 2$), $U_{(n, m)} = \langle x, y :  x^{2n} = y^m = 1, x^{-1}yx = y^{-1}\rangle$   (for   $m \geq 2$   and $n \geq 2$) and  $G(p, m, n) = \langle x, y : x^{p^m} = y^{p^n} = [x, y]^p = 1, [x, [x, y]] = [y, [x, y]] = 1\rangle$ (for  any prime $p$,  $m \geq 1$ and $n \geq 1$).  Following Salahshour and Ashrafi \cite{sA2020}, Bhowal and Nath in their recent paper \cite{bN20} have obtained various spectra and energies (along with several consequences) of  $\mathcal{CCC}(G)$ if $G = D_{2n}, Q_{4m}, SD_{8n}, V_{8n}$ and $U_{(n, m)}$. 
 
  In this paper we compute   spectrum, Laplacian spectrum, signless Laplacian spectrum  and their corresponding energies (i.e., energy Laplacian energy and signless Laplacian energy) of commuting conjugacy class graph of  $G(p, m, n)$. As a consequence we see that  $\mathcal{CCC}(G)$ is super integral \cite{dN2018,hS74} if $G = G(p, m, n)$. Also, it satisfies E-LE Conjecture of Gutman et al. \cite{bN20,gavbr}. Finally, comparing  various energies of $\mathcal{CCC}(G)$, we characterize $G(p, m, n)$  such that its commuting conjugacy class graph is hyperenergetic, L-hyperenergetic or Q-hyperenergetic (see \cite{Walikar-99,Gutman-99,Gong-15,Tura-17,Fasfous-20}). 
 The reason of excluding the group $G(p, m, n)$ in \cite{bN20} is the different nature of commuting conjugacy class graph of $G(p, m, n)$. More precisely, if $G = G(p, m, n)$ then $\mathcal{CCC}(G)$ is of the form $m_1K_{n_1} \sqcup m_2K_{n_2} \sqcup m_3K_{n_3}$ while $\mathcal{CCC}(G) = m_1K_{n_1} \sqcup m_2K_{n_2}$ if $G = D_{2n}, Q_{4m}, SD_{8n}, V_{8n}$ or $U_{(n, m)}$, where $m_iK_{n_i}$ denotes disjoint union of $m_i$ copies of complete graphs $K_{n_i}$ for $i = 1, 2, 3$.

 If $\mathcal{T} = m_1K_{n_1} \sqcup m_2K_{n_2} \sqcup m_3K_{n_3}$ then it is noteworthy that
\begin{equation} \label{prethm1-001}
\spec({\mathcal{T}}) = \left\{(-1)^{\sum_{i=1}^3 m_i(n_i - 1)}, (n_1 - 1)^{m_1},  (n_2 - 1)^{m_2},  (n_3 - 1)^{m_3}\right\} 
\end{equation}
\begin{equation} \label{prethm1-002}
\L-spec({\mathcal{T}}) = \left\{0^{m_1 + m_2 + m_3}, n_1^{m_1(n_1 - 1)}, n_2^{m_2(n_2 - 1)},  n_3^{m_3(n_3 - 1)}\right\} \text{ and}
\end{equation}
\begin{align}\label{prethm1-003}
\Q-spec({\mathcal{T}}) = \Big\lbrace(2n_1 - 2)^{m_1}, (n_1 - 2)^{m_1(n_1 - 1)}, (2n_2 - 2)^{m_2},& (n_2 - 2)^{m_2(n_2 - 1)}, \\ \nonumber 
& (2n_3 - 2)^{m_3}, (n_3 - 2)^{m_3(n_3 - 1)}\Big\rbrace,
\end{align}
where  $\spec({\mathcal{T}}), \L-spec({\mathcal{T}})$ and $\Q-spec({\mathcal{T}})$ denote the spectrum, Laplacian spectrum and signless Laplacian spectrum of $\mathcal{T}$.  
Recall that $\spec({\mathcal{T}}), \L-spec({\mathcal{T}})$ and $\Q-spec({\mathcal{T}})$ contain eigenvalues with the multiplicities (written as exponents) of $A({\mathcal{T}})$, $L({\mathcal{T}})  := D({\mathcal{T}}) - A({\mathcal{T}})$ and    $Q({\mathcal{T}}) := D({\mathcal{T}}) + A({\mathcal{T}})$ respectively, where $A({\mathcal{T}})$  and $D({\mathcal{T}})$ are adjacency and degree matrix of $\mathcal{T}$ respectively. Also, energy ($E({\mathcal{T}})$), Laplacian energy ($LE({\mathcal{T}})$) and signless Laplacian energy ($LE^+({\mathcal{T}})$) are defined as follows:
\begin{equation}\label{energy}
E({\mathcal{T}})=\sum_{x \in \spec({\mathcal{T}})}|x|,
\end{equation}
\begin{equation}\label{L-energy}
LE({\mathcal{T}})= \sum_{x\in \L-spec({\mathcal{T}})}\left|x- \frac{2|e(\mathcal{T})|}{|V(\mathcal{T})|}\right|,
\end{equation}
\begin{equation}\label{Q-energy}
LE^+({\mathcal{T}})= \sum_{x\in \Q-spec({\mathcal{T}})}\left|x- \frac{2|e(\mathcal{T})|}{|V(\mathcal{T})|}\right|,
\end{equation}
where $V({\mathcal{T}})$ is the set of vertices  and $e({\mathcal{T}})$ is the set of   edges of ${\mathcal{T}}$ respectively.

\section{Results} 
We first compute  various spectra and energies of commuting conjugacy class graph of the group $G(p, m, n)$.

\begin{theorem}\label{G(p,n,m)}
If $G = G(p, m, n)$ then 
\begin{enumerate}
\item $\spec(\mathcal{CCC}(G)) =  \Big{\{}(-1)^{p^{m + n} - p^{m + n - 2} - p^n + p^{n - 1} - 2}, (p^m - p^{m - 1} - 1)^{p^{n - 1}(p - 1)},$

\indent \hspace{3cm}  $(p^{m + n - 1} - p^{m + n - 2} - 1)^2\Big{\}}$ and

 $E(\mathcal{CCC}(G))= 2 (p^{m + n} - p^{m + n - 2} - p^n + p^{n - 1} - 2)$.

\item $\L-spec(\mathcal{CCC}(G)) = \Big{\{}0^{p^n - p^{n - 1} + 2}, ( p^{m - 1}(p - 1))^{ p^{n - 2}(p - 1) (p^{m + 1}- p^m - p)}$,\\  \indent \hspace{3cm}  $(p^{m + n - 2}(p - 1) )^{2((p - 1) p^{m + n - 2} - 1)}\Big{\}}$
and 

\noindent $LE(\mathcal{CCC}(G))= \begin{cases}
& \frac{2 (p^{n + 1} - p^n + 2 p) (2 p^{ m + n + 1} - 2 p^{m + n} + p^{m + 3} - 2 p^{m + 2} + p^{m + 1} - p^3-p^2)}{p^3 (p + 1)},\\
&\hspace{3.5cm} \text{ if } n = 1, p \geq 2, m \geq 1; \text{ or } n = 2, p = 2, m = 1\\

&\frac{4}{p^4 (p + 1)}\big{(}p^{2 m + 2 n + 3} - 3 p^{2(m + n + 1)} + 3 p^{2 m + 2 n + 1} - p^{2(m + n)} - p^{2 m + n + 4}\\ 
&\quad\quad\quad\quad  + 3 p^{2 m + n + 3} - 3 p^{2 m + n + 2} + p^{2 m + n + 1} + 2 p^{m + n + 3} - 2 p^{m + n + 2}\\
&\quad\quad\quad\quad   + p^{m + 5} - 2 p^{m + 4} + p^{m + 3} - p^5 - p^4 \big{)}, \text{ \hspace{1.5cm} otherwise.}
\end{cases}$

\item $\Q-spec(\mathcal{CCC}(G)) = \Big{\{}(2 p^m - 2 p^{m - 1} - 2)^{p^{n - 1}(p - 1)}, (p^m - p^{m - 1} - 2)^{p^{n - 2}(p - 1) (p^{m + 1} - p^m - p)},\\
\indent \hspace{3cm}  (2 p^{m + n - 1} - 2 p^{m + n - 2} - 2)^2, (p^{m + n - 1} - p^{m + n - 2} - 2)^{2(p^{m + n - 1} - p^{m + n - 2} - 1)}\Big{\}}$
and 

\noindent $LE^+(\mathcal{CCC}(G))= \left\{
\begin{array}{ll}
 2(p^{m + 1} - p^{m - 1} - p - 1), & \text{ if } n=1, p\geq 2, m\geq 1 \\
\frac{2}{3}(7.2^m - 6),  & \text{ if }  n=2, p=2, m\leq 2\\
\frac{2}{3} (4^m + 2^m - 6),  & \text{ if }  n=2, p=2, m\geq 3\\
\frac{4p^{2 m + n - 4}}{p + 1} (p - 1)^3  (p^n - p), & \text{ if }  n = 2, p\geq 3, m\geq 1;\\ &\text{ or } n \geq 3, p\geq 2, m\geq 1.
\end{array}
\right.$
\end{enumerate}
\end{theorem}
\begin{proof}
By \cite[Proposition 2.6]{sA2020} we have 
\[
\mathcal{CCC}(G) = (p^n - p^{n - 1})K_{p^{m - n}(p^n - p^{n - 1})} \sqcup K_{p^{n - 1}(p^m - p^{m - 1})} \sqcup K_{p^{m - 1}(p^n - p^{n - 1})}.
\]
Let $m_1 = p^n - p^{n - 1}$, 
$m_2 = 1$,
$m_3 = 1$,
$n_1 = p^{m - n}(p^n - p^{n - 1})$,
$n_2 = p^{n - 1}(p^m - p^{m - 1})$  and
$n_3 = p^{m - 1}(p^n - p^{n - 1})$. 
Then, by  \eqref{prethm1-001}-\eqref{prethm1-003}, it follows that
\begin{align*}
\spec(\mathcal{CCC}(G)) &= \left\lbrace(-1)^{p^{m + n} - p^{m + n - 2} - p^n + p^{n - 1} - 2}, (p^m - p^{m - 1} - 1)^{p^{n - 1}(p - 1)},\right. \\
&\quad \quad \quad (p^{m + n - 1} - p^{m + n - 2} - 1)^2\Big{\}},
\end{align*}
\begin{align*}
\L-spec(\mathcal{CCC}(G)) = &\Big{\{}0^{p^n - p^{n - 1} + 2}, ( p^{m - 1}(p - 1))^{ p^{n - 2}(p - 1) (p^{m + 1}- p^m - p)}, \\
& \quad \quad(p^{m + n - 2}(p - 1) )^{2((-1 + p) p^{m + n - 2} - 1)} \Big{\}}
\end{align*}
and
\begin{align*}
\Q-spec(\mathcal{CCC}(G)) &= \Big{\{}(2 p^m - 2 p^{m - 1} - 2)^{p^{n - 1}(p - 1)}, (p^m - p^{m - 1} - 2)^{p^{n - 2}(p - 1) (p^{m + 1} - p^m - p)},\\ 
 & \quad \quad (2 p^{m + n - 1} - 2 p^{m + n - 2} - 2)^2, (p^{m + n - 1} - p^{m + n - 2} - 2)^{2(p^{m + n - 1} - p^{m + n - 2} - 1)}\Big{\}}.
\end{align*}
Hence, by \eqref{energy}, we get 
\begin{align*}
E(\mathcal{CCC}(G))& = p^{m + n} - p^{m + n - 2} - p^n + p^{n - 1} - 2 + (p^{n - 1}(p - 1))(p^m - p^{m - 1} - 1)\\
& \hspace{4.9cm} +  2(p^{m + n - 1} - p^{m + n - 2} - 1)\\
& = 2 (p^{m + n} - p^{m + n - 2} - p^n + p^{n - 1} - 2).
\end{align*}
We have $|V(\mathcal{CCC}(G))| = m_1n_1 + m_2n_2 + m_3n_3 = p^{m + n - 2} (p^2 - 1)$ and 
\begin{align*}
|e(\mathcal{CCC}(G))| &= \frac{m_1n_1(n_1 - 1)}{2} + \frac{m_2n_2(n_2 - 1)}{2} + \frac{m_3n_3(n_3 - 1)}{2} \\
&=\frac{p^{m + n - 4}(p - 1)}{2}(2 p^{m + n + 1} - 2 p^{m + n} + p^{m + 3} - 2 p^{m + 2} + p^{m + 1} - p^3 - p^2).
\end{align*} 
Therefore,
\begin{align*}
\frac{2|e(\mathcal{CCC}(G))|}{|V(\mathcal{CCC}(G))|} &= \frac{2 p^{m + n + 1} - 2 p^{m + n} + p^{m + 3} - 2 p^{m + 2} + p^{m + 1} - p^3 - p^2}{p^2 (p + 1)}\\
&=\frac{p^{m + n}(p - 2) + p^{m + 2}(p - 2) + (p^{m + n + 1} - p^3) + (p^{m + 1} - p^2)}{p^2(1 + p)}\geq 0.
\end{align*}
Also,
\begin{align*}
\left| 0 - \frac{2|e(\mathcal{CCC}(G))|}{|V(\mathcal{CCC}(G))|}\right| &= \frac{2|e(\mathcal{CCC}(G))|}{|V(\mathcal{CCC}(G))|}\\
& =\frac{2 p^{m + n + 1} - 2 p^{m + n} + p^{m + 3} - 2 p^{m + 2} + p^{m + 1} - p^3 - p^2}{p^2 (p + 1)},
\end{align*}
\begin{align*}
\left| p^{m - 1}(p - 1)  - \frac{2|e(\mathcal{CCC}(G))|}{|V(\mathcal{CCC}(G))|}\right| &= \left|\frac{-2 p^{m + n + 1} + 2 p^{m + n} + 2 p^{m + 2} - 2 p^{m + 1} + p^3 + p^2}{p^2 + p^3}\right|\\&=\left| 1-\frac{2p^{m-1}(p^{n-1}-1)(p-1)}{p+1} \right|\\ 
&= \left\{\begin{array}{ll}
1-\frac{2p^{m-1}(p^{n-1}-1)(p-1)}{p+1}, &  \text{ if } n=1,  p \geq 2, m\geq 1;\\ 
&\text{ or } n=2, p=2, m=1 \\
\frac{2p^{m-1}(p^{n-1}-1)(p-1)}{p+1} - 1, &  \text{ otherwise }
        \end{array}
    \right.
\end{align*}
and
\begin{align*}
\left| p^{m + n - 2}(p - 1) \right. & -  \left.\frac{2|e(\mathcal{CCC}(G))|}{|V(\mathcal{CCC}(G))|}\right| \\
&= \left| \frac{p^{m + n + 2}- 2 p^{m + n + 1} + p^{m + n} - p^{m + 3} + 2 p^{m + 2} - p^{m + 1}  + p^3 + p^2}{p^2 (p + 1)} \right|\\
&=\left| \frac{p^{m + n}(p - 1)^2 - p^{m + 1}(p - 1)^2 + p^3 + p^2}{p^2(p + 1)} \right|\\
&=\frac{p^{m + n}(p - 1)^2 - p^{m + 1}(p - 1)^2 + p^3 + p^2}{p^2(p + 1)}.
\end{align*}
Now, by \eqref{L-energy},  we have
\begin{align*}
LE(\mathcal{CCC}(G)) &= (p^n - p^{n - 1} + 2) \times \frac{2 p^{m + n + 1} - 2 p^{m + n} + p^{m + 3} - 2 p^{m + 2} + p^{m + 1} - p^3 - p^2}{p^2 (p + 1)} \\
&\quad + (p^{n - 2}(p - 1) (p^{m + 1}- p^m - p)) \times \left(1 - \frac{2p^{m - 1}(p^{n - 1} - 1)(p - 1)}{p + 1}\right)\\
& \quad + 2((p - 1) p^{m + n - 2} - 1) \times \frac{p^{m + n}(p - 1)^2 - p^{m + 1}(p - 1)^2 + p^3 + p^2}{p^2(p + 1)}\\ 
&=\frac{2 (p^{n + 1} - p^n + 2 p) (2 p^{m + n + 1} - 2 p^{m + n} + p^{m + 3} - 2 p^{m + 2} + p^{m + 1} - p^3-p^2)}{p^3 (p + 1)},  
\end{align*}
if $n=1, p \geq 2, m \geq 1$; or $n=2, p=2, m=1$. Otherwise
\begin{align*}
LE(\mathcal{CCC}(G))&=(p^n - p^{n - 1} + 2) \times \frac{2 p^{m + n + 1} - 2 p^{m + n} + p^{m + 3} - 2 p^{m + 2} + p^{m + 1} - p^3 - p^2}{p^2 (p + 1)} \\
&\quad + (p^{n - 2}(p - 1) (p^{m + 1}- p^m - p)) \times \left(\frac{2p^{m - 1}(p^{n - 1} - 1)(p - 1)}{p + 1} - 1\right)\\
& \quad + 2((p - 1) p^{m + n - 2} - 1) \times \frac{p^{m + n}(p - 1)^2 - p^{m + 1}(p - 1)^2 + p^3 + p^2}{p^2(p + 1)}\\ 
&=\frac{4}{p^4 (p + 1)}\big{(}p^{2 m + 2 n + 3} - 3 p^{2(m + n + 1)} + 3 p^{2 m + 2 n + 1} - p^{2(m + n)} - p^{2 m + n + 4}\\ 
&\quad\quad\quad\quad  + 3 p^{2 m + n + 3} - 3 p^{2 m + n + 2} + p^{2 m + n + 1} + 2 p^{m + n + 3} - 2 p^{m + n + 2}\\
&\quad\quad\quad\quad   + p^{m + 5} - 2 p^{m + 4} + p^{m + 3} - p^5 - p^4 \big{)}.
\end{align*}
Again,
\begin{align*}
\left| 2 p^m - 2 p^{m - 1} - 2 \right. & - \left.\frac{2|e(\mathcal{CCC}(G))|}{|V(\mathcal{CCC}(G))|}\right| \\
&= \left| -\frac{2 p^{m + n + 1} - 2 p^{m + n} - p^{m + 3} - 2 p^{m + 2} + 3 p^{m + 1} + p^3 + p^2}{p^3 + p^2} \right|\\
&= \left| \frac{f_{1}(p, m, n)}{p^3 + p^2} \right|,
\end{align*}
where $f_{1}(p, m, n)= -(2 p^{m + n + 1} - 2 p^{m + n} - p^{m + 3} - 2 p^{m + 2} + 3 p^{m + 1} + p^3 + p^2)$. For $n=1, p\geq 2, m\geq 1$, we have $f_{1}(p, m, 1)= p (p + 1) (p^{m + 1} - p^m - p)\geq 0$. For $n=2, p\geq 2, m\geq 1$, we have
\begin{align*}
f_{1}(p,m,2)= -(p^{m + 3} - 4 p^{m + 2} + 3 p^{m + 1} + p^3 + p^2) = -(p^{m + 1}(p - 1)(p - 3) + p^3 + p^2).
\end{align*}
So, $f_{1}(2,m,2)= 2 (2^m - 6)> 0$ for $m\geq 3$ and $f_{1}(2,m,2) < 0$  for $m = 1, 2$. Also, $f_{1}(p, m, 2)\leq 0$ for  $p\geq 3$ and $m\geq 1$. For $n\geq 3, p\geq 2, m\geq 1$, we have
\begin{align*}
f_{1}(p, m, n)= -((p - 1)(p^{m + 2}(p^{n - 2} - 1) + p^{m + 1}(p^{n - 1} - 3)) + p^3 + p^2)\leq 0.
\end{align*}
Therefore 
\begin{align*}
\left| (2 p^m -  \right. & \left. 2 p^{m - 1} - 2)  - \frac{2|e(\mathcal{CCC}(G))|}{|V(\mathcal{CCC}(G))|}\right|\\ &= \left\{
        \begin{array}{ll}
             -\frac{2 p^{m + n + 1} - 2 p^{m + n} - p^{m + 3} - 2 p^{m + 2} + 3 p^{m + 1} + p^3 + p^2}{p^3 + p^2}, &  \text{ if } n=1, p\geq 2, m\geq 1;\\ & \text{ or }n=2, p=2, m\geq 3 \\
             \frac{2 p^{m + n + 1} - 2 p^{m + n} - p^{m + 3} - 2 p^{m + 2} + 3 p^{m + 1} + p^3 + p^2}{p^3 + p^2}, &  \text{ otherwise. }
        \end{array}
    \right.
\end{align*}
We have
\begin{align*}
\left| (p^m  - p^{m - 1} - 2) - \frac{2|e(\mathcal{CCC}(G))|}{|V(\mathcal{CCC}(G))|}\right| &= \left| -\frac{2 p^{m + n + 1} - 2 p^{m + n} - 2 p^{m + 2} + 2 p^{m + 1} + p^3 + p^2}{p^3 + p^2} \right|\\
&=\left| -\frac{2 p^{m + n}(p-1) - 2 p^{m + 1}(p-1) + p^3 + p^2}{p^3 + p^2} \right| \\
&= \frac{2 p^{m + n}(p-1) - 2 p^{m + 1}(p-1) + p^3 + p^2}{p^3 + p^2},
\end{align*}
\begin{align*}
\left|2 p^{m + n - 1} - 2 p^{m + n - 2} - 2 \right. &- \left.\frac{2|e(\mathcal{CCC}(G))|}{|V(\mathcal{CCC}(G))|}\right|\\
&= \left| -\frac{-2 p^{m + n + 1} + 2 p^{m + n} + p^{m + 2} - 2 p^{m + 1} + p^m + p^2 + p}{p + p^2} \right|\\
&=\left|\frac{(p - 1)(2p^{m + n - 1} - p^{m + 1} + p^m)}{p(p + 1)} -1 \right|\\
&=\left| (p - 1)\frac{(p^{m + n - 1} - p^m) + (p^{m + n - 1} + p^{m - 1})}{p + 1} - 1 \right|\\
&= (p - 1)\frac{(p^{m + n - 1} - p^m) + (p^{m + n - 1} + p^{m - 1})}{p + 1} - 1
\end{align*}
and
\begin{align*}
\left|p^{m + n - 1} \right. & - \left. p^{m + n - 2} - 2   -  \frac{2|e(\mathcal{CCC}(G))|}{|V(\mathcal{CCC}(G))|}\right| \\
&= \left| -\frac{ - p^{m + n + 2} +  2 p^{m + n + 1} - p^{m + n} + p^{m + 3} - 2 p^{m + 2} + p^{m + 1} + p^3 + p^2}{p^2 (p + 1)} \right|\\ 
&= \left| -1 + \frac{p^{m - 1}(p^{n - 1} - 1)(p - 1)^2}{p + 1} \right|\\
&= \left\{\begin{array}{ll}
1 - \frac{p^{m - 1}(p^{n - 1} - 1)(p - 1)^2}{p + 1}, & \text{if } n=1, p\geq 2, m\geq 1;\\ & n=2, p=2, m\leq 2 \\
-1 + \frac{p^{m - 1}(p^{n - 1} - 1)(p - 1)^2}{p + 1}, &  \text{otherwise.}
\end{array}
\right.
\end{align*}
By \eqref{Q-energy},  we have
\begin{align*}
LE^+&(\mathcal{CCC}(G))\\
&= (p^{n - 1}(p - 1))\times \left( -\frac{2 p^{m + n + 1} - 2 p^{m + n} - p^{m + 3} - 2 p^{m + 2} + 3 p^{m + 1} + p^3 + p^2}{p^3 + p^2}\right)\\
& \quad + p^{n - 2}(p - 1) (p^{m + 1} - p^m - p) \times \left(\frac{2 p^{m + n}(p-1) - 2 p^{m + 1}(p-1) + p^3 + p^2}{p^3 + p^2}\right)\\
& \quad + 2 \times \left((p-1)\frac{(p^{m+n-1}-p^m)+(p^{m+n-1}+p^{m-1})}{p+1} - 1\right)\\
& \quad + 2(p^{m + n - 1} - p^{m + n - 2} - 1)\times \left( 1-\frac{p^{m-1}(p^{n-1}-1)(p-1)^2}{p+1}\right)\\
&= 2(p^{m + 1} - p^{m - 1} - p - 1),
\end{align*}
if $n=1, p\geq 2, m\geq 1$.
If $n=2, p=2, m\leq 2$ then
\begin{align*}
LE^+&(\mathcal{CCC}(G))\\
&=(p^{n - 1}(p - 1))\times \left(\frac{2 p^{m + n + 1} - 2 p^{m + n} - p^{m + 3} - 2 p^{m + 2} + 3 p^{m + 1} + p^3 + p^2}{p^3 + p^2}\right)\\
& \quad + p^{n - 2}(p - 1) (p^{m + 1} - p^m - p) \times \left(\frac{2 p^{m + n}(p-1) - 2 p^{m + 1}(p-1) + p^3 + p^2}{p^3 + p^2}\right)\\
& \quad + 2 \times \left((p-1)\frac{(p^{m+n-1}-p^m)+(p^{m+n-1}+p^{m-1})}{p+1} - 1\right)\\
& \quad + 2(p^{m + n - 1} - p^{m + n - 2} - 1)\times \left( 1-\frac{p^{m-1}(p^{n-1}-1)(p-1)^2}{p+1}\right)\\
&= \frac{2}{3}(7.2^m - 6).
\end{align*}
If $n=2, p=2, m\geq 3$ then 
\begin{align*}
LE^+&(\mathcal{CCC}(G))\\
&=(p^{n - 1}(p - 1))\times \left(-\frac{2 p^{m + n + 1} - 2 p^{m + n} - p^{m + 3} - 2 p^{m + 2} + 3 p^{m + 1} + p^3 + p^2}{p^3 + p^2}\right)\\
& \quad + p^{n - 2}(p - 1) (p^{m + 1} - p^m - p) \times \left(\frac{2 p^{m + n}(p-1) - 2 p^{m + 1}(p-1) + p^3 + p^2}{p^3 + p^2}\right)\\
& \quad + 2 \times \left((p-1)\frac{(p^{m+n-1}-p^m)+(p^{m+n-1}+p^{m-1})}{p+1} - 1\right)\\
& \quad + 2(p^{m + n - 1} - p^{m + n - 2} - 1)\times \left(\frac{p^{m-1}(p^{n-1}-1)(p-1)^2}{p+1} - 1\right)\\
&= \frac{2}{3} (4^m + 2^m - 6).
\end{align*}
If $n\geq 3, p\geq 2, m\geq 1$ then 
\begin{align*}
LE^+&(\mathcal{CCC}(G))\\
&=(p^{n - 1}(p - 1))\times \left(\frac{2 p^{m + n + 1} - 2 p^{m + n} - p^{m + 3} - 2 p^{m + 2} + 3 p^{m + 1} + p^3 + p^2}{p^3 + p^2}\right)\\
& \quad + p^{n - 2}(p - 1) (p^{m + 1} - p^m - p) \times \left(\frac{2 p^{m + n}(p-1) - 2 p^{m + 1}(p-1) + p^3 + p^2}{p^3 + p^2}\right)\\
& \quad + 2 \times \left((p-1)\frac{(p^{m+n-1}-p^m)+(p^{m+n-1}+p^{m-1})}{p+1} - 1\right)\\
& \quad + 2(p^{m + n - 1} - p^{m + n - 2} - 1)\times \left(\frac{p^{m-1}(p^{n-1}-1)(p-1)^2}{p+1} - 1\right)\\
& = \frac{4p^{2 m + n - 4}}{p + 1} (p - 1)^3  (p^n - p).
\end{align*}
This completes the proof.
\end{proof}
The next two results give comparison between various energies of  commuting conjugacy class graph of $G(p, m, n)$ and also with energies of complete graphs   $K_{|V(G(p,m,n))|}$.
\begin{theorem}\label{cG(p,m,n)}
Let $G = G(p, m, n)$.
\begin{enumerate}
\item If $n=1, p\geq 2$ and $m\geq 1$ then $E(\mathcal{CCC}(G)) = LE^+(\mathcal{CCC}(G)) = LE(\mathcal{CCC}(G)).
$
\item If $n=2, p = 2$ and  $m = 1$ then $E(\mathcal{CCC}(G)) < LE^+(\mathcal{CCC}(G)) = LE(\mathcal{CCC}(G)).$
\item If $n=2$, $p = 2$ and $m = 2$ then $LE^+(\mathcal{CCC}(G)) < E(\mathcal{CCC}(G)) < LE(\mathcal{CCC}(G)).$ 
\item If $n=2$, $p = 2$, $m \geq 3$; $n=2, p\geq 3$,  $m \geq 1$; or $n\geq 3, p\geq 2$,   $m\geq 1$ then  
$$E(\mathcal{CCC}(G)) < LE^+(\mathcal{CCC}(G)) < LE(\mathcal{CCC}(G)).$$
\end{enumerate}
\end{theorem}
\begin{proof}
We shall proof the result by considering the following cases.

\noindent \textbf{Case 1.} $n = 1$,  $p \geq 2$ and $m \geq 1$. 

By Theorem \ref{G(p,n,m)}, we have
\[
E(\mathcal{CCC}(G)) = LE(\mathcal{CCC}(G)) = LE^+(\mathcal{CCC}(G)) = 2 (p^{m + 1} - p^{m - 1} - p - 1).
\]

\noindent \textbf{Case 2.} $n = 2, p = 2$ and $m \geq 1$.

If $n = 2,  p = 2$ and $m=1$ then, by Theorem \ref{G(p,n,m)}, we have
\begin{align*}
LE^+(\mathcal{CCC}(G)) - E(\mathcal{CCC}(G)) &= \frac{2}{3}(7.2^m - 6) - 2 (p^{m + n} - p^{m + n - 2} - p^n + p^{n - 1} - 2)\\ 
& = \frac{16}{3} - 4 = \frac{4}{3} > 0 
\end{align*}
and
\[
LE(\mathcal{CCC}(G)) = LE^+(\mathcal{CCC}(G)) =  \frac{16}{3}.
\]
Therefore, $ E(\mathcal{CCC}(G)) <  LE^+(\mathcal{CCC}(G)) = LE(\mathcal{CCC}(G))$.

If $n=2,p=2$ and $m=2$ then, by Theorem \ref{G(p,n,m)}, we have
\begin{align*}
LE(\mathcal{CCC}&(G)) - E(\mathcal{CCC}(G))\\
& = \frac{4}{p^4 (p + 1)}\big{(}p^{2 m + 2 n + 3} - 3 p^{2(m + n + 1)} + 3 p^{2 m + 2 n + 1} - p^{2(m + n)} - p^{2 m + n + 4} + 3 p^{2 m + n + 3}\\ 
&\quad  - 3 p^{2 m + n + 2} + p^{2 m + n + 1} + 2 p^{m + n + 3} - 2 p^{m + n + 2} + p^{m + 5} - 2 p^{m + 4} + p^{m + 3} - p^5 - p^4 \big{)}\\
&\quad  - 2 (p^{m + n} - p^{m + n - 2} - p^n + p^{n - 1} - 2) = 20 - 16 = 4 > 0
\end{align*}
and
\begin{align*}
LE^+(\mathcal{CCC}(G)) - E(\mathcal{CCC}(G))  
&= \frac{2}{3}(7.2^m - 6) -  2 (p^{m + n} - p^{m + n - 2} - p^n + p^{n - 1} - 2)\\
& = \frac{44}{3} - 16 = -\frac{4}{3} < 0.
\end{align*}
Therefore, $LE^+(\mathcal{CCC}(G)) < E(\mathcal{CCC}(G)) < LE(\mathcal{CCC}(G))$.

If $n = 2, p=2$ and $m\geq 3$ then, by Theorem \ref{G(p,n,m)}, we have
\begin{align*}
LE^+(\mathcal{CCC}(G)) &- E(\mathcal{CCC}(G)) \\ 
& = \frac{2}{3} (4^m + 2^m - 6) - 2 (p^{m + n} - p^{m + n - 2} - p^n + p^{n - 1} - 2) \\
& = \frac{2}{3} (4^m + 2^m - 6) - 2 (3.2^m - 4) \\
& = -\frac{2}{3} (2^{2m} - 2^{m + 3} + 6) =  -\frac{2}{3} (2^{m}(2^{m}- 8) + 6) < 0
\end{align*}
and
\begin{align*}
LE&(\mathcal{CCC}(G)) - LE^+(\mathcal{CCC}(G))\\ 
&= \frac{4}{p^4 (p + 1)}\big{(}p^{2 m + 2 n + 3} - 3 p^{2(m + n + 1)} + 3 p^{2 m + 2 n + 1} - p^{2(m + n)} - p^{2 m + n + 4} + 3 p^{2 m + n + 3}\\ 
&\quad  - 3 p^{2 m + n + 2} + p^{2 m + n + 1} + 2 p^{m + n + 3} - 2 p^{m + n + 2} + p^{m + 5} - 2 p^{m + 4} + p^{m + 3} - p^5 - p^4 \big{)}  \\
&\quad - \frac{2}{3} (4^m + 2^m - 6)\\
&= \frac{2}{3} (4^m + 5.2^m - 6) - \frac{2}{3} (4^m + 2^m - 6) \\
&= \frac{2^{m + 3}}{3} > 0.
\end{align*}
Therefore, $ E(\mathcal{CCC}(G)) < LE^+(\mathcal{CCC}(G)) < LE(\mathcal{CCC}(G))$.

\noindent \textbf{Case 3.} $n = 2, p \geq 3$, $m \geq1$; or $n \geq 3, p \geq 2$,  $m \geq1$. 

By Theorem \ref{G(p,n,m)}, we have
\begin{align*}
LE^+(\mathcal{CCC}(G)) &- E(\mathcal{CCC}(G)) \\
& = \frac{4p^{2 m + n - 4}}{p + 1} (p - 1)^3  (p^n - p) - 2 (p^{m + n} - p^{m + n - 2} - p^n + p^{n - 1} - 2)\\
& = \frac{2}{p^4 (p + 1)}(2 p^{2 m + 2 n + 3} - 6 p^{2 (m + n + 1)} + 6 p^{2 m + 2 n + 1} - 2 p^{2 (m + n)} - 2 p^{2 m + n + 4}\\
&\quad + 6 p^{2 m + n + 3} - 6 p^{2 m + n + 2} + 2 p^{2 m + n + 1} - p^{m + n + 5} - p^{m + n + 4} + p^{m + n + 3}\\
&\quad + p^{m + n + 2} + p^{n + 5} - p^{n + 3} + 2 p^5 + 2 p^4 )\\
&:= f(p,m,n).
\end{align*}
Therefore, for $n = 2$, we have
\begin{align*}
f(p,m,2)&= \frac{2}{p (p + 1)}(2 p^{2 m + 4} - 8 p^{2 m + 3} + 12 p^{2 m + 2} - 8 p^{2 m + 1} + 2 p^{2 m} - p^{m + 4}\\
&\quad - p^{m + 3} + p^{m + 2} + p^{m + 1} + p^4 + p^2 + 2 p)
\end{align*}
and so $f(3,m,2) = \frac{16}{3} (9^m - 3^{m + 1} + 3) > 0$ for $m\geq 1$.
Also,
\begin{align*}
f(p,m,2)
& = \frac{2}{p (p + 1)}(p^{2 m + 3}(2p - 9) + (p^{2m + 3} - p^{m + 4}) + (p^{2m + 2} -  p^{m + 3}) + p^{2m + 1}(11p - 8)\\
& \quad + 2 p^{2 m} + p^{m + 2} + p^{m + 1} + p^4 + p^2 + 2 p ) > 0,
\end{align*}
if $m\geq 1$ and $p\geq 5$. For $p=2$, we have
\begin{align*}
f(2,m,n) 
& =  \frac{1}{12} (2^{m + n}\left((2^{n - 1} - 1)2^{m + 1} - 18) + 3.2^{n + 2} + 48\right) > 0, 
\end{align*}
if $n\geq 4, p = 2, m\geq 1$. For $n = 3, p = 2, m\geq 1$ we have $f(2,m,3) = 48 (4^m - 3.2^m + 3) = 48 ((2^m - 2)^2 + 2^m - 1) > 0$.

For $p\geq 3, n\geq 3, m\geq 1$, we have
\begin{align*}
f(p,m,n) &= \frac{2}{p^4 (p + 1)}(2 p^{2 m + 2 n + 2} (p - 3) + (p^{2 m + 2 n + 1} - p^{m + n + 5}) + 2 p^{2 m + n + 1} + 4 p^{2 m + n + 3}\\
&\quad + (p^{2 m + 2 n + 1} - 2 p^{2 m + 2 n}) + (p^{2 m + 2 n + 1} - 2 p^{2 m + n + 4}) + 2 p^{2 m + 2 n + 1}\\
&\quad + 2 p^{2 m + n + 2} (p - 3) + (p^{2 m + 2 n + 1} - p^{m + n + 4}) + p^{m + n + 3}\\
&\quad + p^{m + n + 2} + p^{n + 3} (p^2 - 1) + 2 p^5 + 2 p^4 \big{)} > 0.
\end{align*}
 Therefore, $LE^+(\mathcal{CCC}(G)) > E(\mathcal{CCC}(G))$ if $n = 2, p \geq 3$,  $m \geq1$; or $n \geq 3, p \geq 2$, $m \geq1$.
 
Again,
\begin{align*}
LE(\mathcal{CCC}&(G)) - LE^+(\mathcal{CCC}(G))\\
& = \frac{4}{p^4 (p + 1)}\big{(}p^{2 m + 2 n + 3} - 3 p^{2(m + n + 1)} + 3 p^{2 m + 2 n + 1} - p^{2(m + n)} - p^{2 m + n + 4} + 3 p^{2 m + n + 3}\\ 
&\quad  - 3 p^{2 m + n + 2} + p^{2 m + n + 1} + 2 p^{m + n + 3} - 2 p^{m + n + 2} + p^{m + 5} - 2 p^{m + 4} + p^{m + 3}\\
& \quad - p^5 - p^4 \big{)} - \frac{4p^{2 m + n - 4}}{p + 1} (p - 1)^3  (p^n - p)\\
& = \frac{4 (2 p^{m + n + 1} - 2 p^{m + n} + p^{m + 3} - 2 p^{m + 2} + p^{m + 1} - p^3 - p^2)}{p^2 (p + 1)}\\
& = \frac{4 \big{(}(p^{m + n}(p-2) + (p^{m + n + 1} - p^3) + p^{m + 2}(p-2) + p^{m + 1} - p^2)\big{)}}{p^2 (1 + p)} > 0,
\end{align*}
if  $p\geq 2, n\geq 2, m\geq 1$. Therefore, $LE(\mathcal{CCC}(G)) > LE^+(\mathcal{CCC}(G))$, if  $n = 2, p \geq 3$,  $m \geq1$; or $n \geq 3, p \geq 2$,  $m \geq1$. Hence,
$$E(\mathcal{CCC}(G)) < LE^+(\mathcal{CCC}(G))<LE(\mathcal{CCC}(G)),$$ if $n = 2, p \geq 3$,  $m \geq 1$; or $n \geq 3, p \geq 2$,  $m \geq1$. This completes the proof.
\end{proof}

\begin{theorem}\label{hyper-G(p,n,m)}
Let $G= G(p,m,n)$.
\begin{enumerate}
\item If $n = 1$,  $p \geq 2$,  $m \geq 1$; $n=2$, $p=2$,  $m=1, 2$ then $\mathcal{CCC}(G)$ is neither hyperenergetic, borderenergetic, L-hyperenergetic, L-borderenergetic,   Q-hyperenergetic, nor Q-borderenergetic.

\item If $n=2$, $p=2$, $m = 3$; or $n=3$, $p=2$, $m = 1$ then  $\mathcal{CCC}(G)$ is L-hyperenergetic but neither hyperenergetic,   borderenergetic, L-borderenergetic, Q-hyperenergetic nor Q-borderenergetic.

\item If $n=2$, $p=2$,  $m \geq 4$; $n = 2, p \geq 3, m\geq 1$; $n=3, p = 2, m\geq 2$; or $n\geq 4, p\geq 2, m\geq 1$ then  $\mathcal{CCC}(G)$ is L-hyperenergetic  and Q-hyperenergetic but neither hyperenergetic,   borderenergetic, L-borderenergetic nor Q-borderenergetic.
\end{enumerate}

\end{theorem}
\begin{proof}
By \cite[Proposition 2.6]{sA2020}, we have 
\[
\mathcal{CCC}(G) = (p^n - p^{n - 1})K_{p^{m - n}(p^n - p^{n - 1})} \sqcup K_{p^{n - 1}(p^m - p^{m - 1})} \sqcup K_{p^{m - 1}(p^n - p^{n - 1})}.
\]
Therefore, $|V(\mathcal{CCC}(G))|=  p^{m + n - 2} (p^2 - 1)$ and so
\[
 E(K_{|V(\mathcal{CCC}(G))|})= LE^+(K_{|V(\mathcal{CCC}(G))|}) = LE(K_{|V(\mathcal{CCC}(G))|}) = 2(p^{m + n} - p^{m + n - 2} - 1),
 \]
noting that $E(K_n)= LE(K_n)= LE^+(K_n) = 2(n-1)$.

\vspace{1.5cm}

We shall prove the result by considering the following cases.
 
\noindent \textbf{Case 1.} $n = 1$,  $p \geq 2$ and $m \geq 1$.
 
If $n=1, p \geq 2$ and $m\geq 1$    then, by   Theorem \ref{cG(p,m,n)}, we get
 \[
 LE^+(\mathcal{CCC}(G)) = E(\mathcal{CCC}(G)) = LE(\mathcal{CCC}(G)). 
 \]
By Theorem \ref{cG(p,m,n)}, we also have
\begin{align*}
LE(\mathcal{CCC}(G)) &- 2(p^{m + n} - p^{m + n - 2} - 1) \\ 
&=  \frac{2 (p^{n + 1} - p^n + 2 p) (2 p^{m + n + 1} - 2 p^{m + n} + p^{m + 3} - 2 p^{m + 2} + p^{m + 1} - p^3 - p^2)}{p^3 (1 + p)} \\
& \quad - 2(p^{m + n} - p^{m + n - 2} - 1) \\
&= -2p < 0.
\end{align*}
Therefore, $LE(\mathcal{CCC}(G)) < LE(K_{|V(\mathcal{CCC}(G))|})$ and so
$\mathcal{CCC}(G)$ is neither hyperenergetic, borderenergetic, L-hyperenergetic, L-borderenergetic,   Q-hyperenergetic, nor Q-borderenergetic. Thus, if $n = 1$,  $p \geq 2$ and $m \geq 1$ then $\mathcal{CCC}(G)$ is neither hyperenergetic, borderenergetic, L-hyperenergetic, L-borderenergetic,   Q-hyperenergetic, nor Q-borderenergetic.

\noindent \textbf{Case 2.} $n = 2, p = 2$ and $m \geq1$.

If $n=2$, $p=2$ and $m=1$   then, by   Theorem \ref{cG(p,m,n)} and Theorem \ref{G(p,n,m)}, we get
\[
E(\mathcal{CCC}(G)) < LE^+(\mathcal{CCC}(G)) = LE(\mathcal{CCC}(G)) = \frac{16}{3} < 10 =  LE(K_{|V(\mathcal{CCC}(G))|}).
\]
Therefore, $\mathcal{CCC}(G)$ is neither hyperenergetic, borderenergetic, L-hyperenergetic, L-borderenergetic,   Q-hyperenergetic, nor Q-borderenergetic.
 
 If $n=2$, $p=2$ and $m=2$   then, by   Theorem \ref{cG(p,m,n)} and Theorem \ref{G(p,n,m)}, we get
\[
 LE^+(\mathcal{CCC}(G)) < E(\mathcal{CCC}(G)) < LE(\mathcal{CCC}(G)) = 20 < 22 =  LE(K_{|V(\mathcal{CCC}(G))|}).
\]
Therefore,  $\mathcal{CCC}(G)$ is neither hyperenergetic, borderenergetic, L-hyperenergetic, L-borderenergetic,   Q-hyperenergetic, nor Q-borderenergetic.
 
 If $n=2$, $p=2$ and $m\geq 3$    then, by Theorem \ref{G(p,n,m)}, we get   
$E(\mathcal{CCC}(G))=   2 (3.2^m - 4)$ and $E(K_{|V(\mathcal{CCC}(G))|}) = 2 (3.2^m - 1)$ and so  $E(\mathcal{CCC}(G)) - E(K_{|V(\mathcal{CCC}(G))|}) = 2 (3.2^m - 4) - 6.2^m + 2 = -6 < 0$. Therefore, $\mathcal{CCC}(G)$ is neither hyperenergetic nor borderenergetic.
 
By Theorem \ref{G(p,n,m)}, we also get
$
LE(\mathcal{CCC}(G)) = \frac{2}{3} (4^m + 5.2^m - 6)
$ for $n=2$, $p=2$ and $m\geq 3$. Therefore, $LE(\mathcal{CCC}(G)) -  (3.2^m -1) = \frac{2}{3} (4^m - 2^{m + 2} - 3)> 0$ and so $LE(\mathcal{CCC}(G)) > LE(K_{|V(\mathcal{CCC}(G))|})$. Hence, $\mathcal{CCC}(G)$ is L-hyperenergetic but not L-borderenergetic for $n=2$, $p=2$ and $m\geq 3$.

If $n=2$, $p=2$ and $m\geq 3$ then, by Theorem \ref{G(p,n,m)}, we  get $LE^+(\mathcal{CCC}(G)) = \frac{2}{3}(2^m -2)(2^m + 3)$. Therefore, $LE^+(\mathcal{CCC}(G)) - (3.2^m -1) =  \frac{2}{3}(2^m(2^m - 8)  - 3)$ and so $LE^+(\mathcal{CCC}(G)) < LE^+(K_{|V(\mathcal{CCC}(G))|})$ or $LE^+(\mathcal{CCC}(G)) > LE^+(K_{|V(\mathcal{CCC}(G))|})$  according as $m = 3$ or $m \geq 4$. Hence, $\mathcal{CCC}(G)$ is neither  Q-hyperenergetic nor Q-borderenergetic if $m = 3$; and if $m \geq 4$ then $\mathcal{CCC}(G)$ is Q-hyperenergetic but not Q-borderenergetic. Thus, if $n=2$, $p=2$ and $m = 3$ then  $\mathcal{CCC}(G)$ is L-hyperenergetic but neither hyperenergetic,   borderenergetic, L-borderenergetic, Q-hyperenergetic nor Q-borderenergetic; and if $n=2$, $p=2$ and $m \geq 4$ then  $\mathcal{CCC}(G)$ is L-hyperenergetic  and Q-hyperenergetic but neither hyperenergetic,   borderenergetic, L-borderenergetic nor Q-borderenergetic.



\noindent \textbf{Case 3.} $n=2$, $p\geq 3$,  $m\geq 1$; or $n\geq 3$, $p\geq 2$, $m\geq 1$.

By Theorem \ref{G(p,n,m)}, we get
\[
E(\mathcal{CCC}(G))= 2 (p^{m + n} - p^{m + n - 2} - p^n + p^{n - 1} - 2),
\]
\[
LE^+(\mathcal{CCC}(G)) = \frac{4p^{2 m + n - 4}}{p + 1} (p - 1)^3  (p^n - p)
\]
and
\begin{align*}
LE(\mathcal{CCC}(G)) &= \frac{4}{p^4 (p + 1)}\big{(}p^{2 m + 2 n + 3} - 3 p^{2(m + n + 1)} + 3 p^{2 m + 2 n + 1} - p^{2(m + n)} - p^{2 m + n + 4}\\ 
&\quad  + 3 p^{2 m + n + 3} - 3 p^{2 m + n + 2} + p^{2 m + n + 1} + 2 p^{m + n + 3} - 2 p^{m + n + 2}\\
&\quad  + p^{m + 5} - 2 p^{m + 4} + p^{m + 3} - p^5 - p^4 \big{)}.
\end{align*}
We have
\begin{align*}
E(\mathcal{CCC}(G)) - 2(p^{m + n} - p^{m + n - 2} - 1) & = 2 (p^{m + n} - p^{m + n - 2} - p^n + p^{n - 1} - 2) \\
 & \quad - 2(p^{m + n} - p^{m + n - 2} - 1) \\ & = - 2 p^n + 2 p^{n - 1} - 2 < 0.
\end{align*}
Therefore, $\mathcal{CCC}(G)$ is neither hyperenergetic nor borderenergetic. Also,
\begin{align*}
LE^+&(\mathcal{CCC}(G)) - 2(p^{m + n} - p^{m + n - 2} - 1) \\
&= \frac{4p^{2 m + n - 4}}{p + 1} (p - 1)^3  (p^n - p)  - 2(p^{m + n} - p^{m + n - 2} - 1)\\
&= \frac{2}{p^4 (p + 1)}(2 p^{2 m + 2 n + 3} - 6 p^{2 m + 2 n + 2} + 6 p^{2 m + 2 n + 1} - 2 p^{2 m + 2 n} - 2 p^{2 m + n + 4} + 6 p^{2 m + n + 3}\\
&\quad - 6 p^{2 m + n + 2} + 2 p^{2 m + n + 1} - p^{m + n + 5} - p^{m + n + 4} + p^{m + n + 3} + p^{m + n + 2} + p^5 + p^4)\\
& := f_{1}(p,m,n).
\end{align*}
Now, for  $p = 2,$ $n \geq 3$ and $m\geq 2$, we have
\begin{align*}
f_{1}(2,m,n) 
&= \frac{1}{12}\left(2^{m + n}(2^m(2^n - 2)-18) + 24\right) > 0.
\end{align*}
For $p=2, n=3, m=1$ we have $f_{23}(2,1,3)= -6 < 0$. For $n=2, p=3$ we have 
\[
f_{1}(3,m,2) = 4^{m + 1} - 3.2^{m + 2} + 2 > 0,
\]
if $m\geq 2$. For $n=2, p=3$ and  $m=1$ we have $f_{1}(3,1,2) = 2 > 0$. For $n=2$ we have
\begin{align*}
f_{1}(p,m,2) &= \frac{2}{p (p + 1)} (2 p^{2 m + 4} - 8 p^{2 m + 3} + 12 p^{2 m + 2} - 8 p^{2 m + 1} + 2 p^{2 m} - p^{m + 4} - p^{m + 3} + p^{m + 2} \\
&\quad  + p^{m + 1} + p^2 + p)\\
& = \frac{2}{p (p + 1)} (2p^{2 m + 3}(p - 5) + ( p^{2m + 3} - p^{m + 4}) + (p^{2m + 3} - p^{m + 3}) + (12 p^{2 m + 2} - 8 p^{2 m + 1})\\
&\quad + 2 p^{2m} + p^{m + 2} + p^{m + 1} + p^2 + p\big{)} > 0,
\end{align*}
if $p\geq 5,m\geq 1$. For $n\geq 3, p\geq 3, m\geq 1$ we have
\begin{align*}
f_{1}(p,m,n) &=  \frac{2}{p^4 (p + 1)}(2 p^{2 m + 2 n + 3} - 6 p^{2 m + 2 n + 2} + 6 p^{2 m + 2 n + 1} - 2 p^{2 m + 2 n} - 2 p^{2 m + n + 4} + 6 p^{2 m + n + 3}\\
&\quad - 6 p^{2 m + n + 2} + 2 p^{2 m + n + 1} - p^{m + n + 5} - p^{m + n + 4} + p^{m + n + 3} + p^{m + n + 2} + p^5 + p^4)\\
&= \frac{2}{p^4 (p + 1)}((2 p^{2 m + 2 n + 3} - 6 p^{2 (m + n + 1)}) + (2 p^{2 m + 2 n + 1} - 2 p^{2 m + n + 4}) + (6 p^{2 m + n + 3}\\
&\quad - 6 p^{2 m + n + 2}) + 2 p^{2 m + n + 1} + (p^{2 m + 2 n + 1} - p^{m + n + 5}) + ( 2 p^{2 m + 2 n + 1} - p^{m + n + 4})\\
&\quad + (p^{2 m + 2 n + 1} - 2 p^{2 (m + n)}) + p^{m + n + 3} + p^{m + n + 2} + p^5 + p^4 \big{)} > 0.
\end{align*}
Therefore, $\mathcal{CCC}(G)$ is neither Q-hyperenergetic nor Q-borderenergetic for $n=3, p=2, m=1$. For $n = 2, p \geq 3, m\geq 1$; $n=3, p = 2, m\geq 2$; $n\geq 4, p\geq 2, m\geq 1$; $\mathcal{CCC}(G)$ is Q-hyperenergetic but not Q-borderenergetic.

We have
\begin{align*}
LE&(\mathcal{CCC}(G)) - 2(p^{m + n} - p^{m + n - 2} - 1) \\
&= \frac{4}{p^4 (p + 1)}\big{(}p^{2 m + 2 n + 3} - 3 p^{2(m + n + 1)} + 3 p^{2 m + 2 n + 1} - p^{2(m + n)} - p^{2 m + n + 4} + 3 p^{2 m + n + 3}\\ 
&\quad  - 3 p^{2 m + n + 2} + p^{2 m + n + 1} + 2 p^{m + n + 3} - 2 p^{m + n + 2} + p^{m + 5} - 2 p^{m + 4}\\
&\quad  + p^{m + 3} - p^5 - p^4 \big{)} - 2(p^{m + n} - p^{m + n - 2} - 1) \\
&= \frac{2}{p^4 (p + 1)} (2 p^{2 m + 2 n + 3} - 6 p^{2 (m + n + 1)} + 6 p^{2 m + 2 n + 1} -  2 p^{2 (m + n)} - 2 p^{2 m + n + 4} + 6 p^{2 m + n + 3}\\
&\quad - 6 p^{2 m + n + 2} + 2 p^{2 m + n + 1} - p^{m + n + 5} - p^{m + n + 4} + 5 p^{m + n + 3} - 3 p^{m + n + 2} + 2 p^{m + 5}\\
&\quad - 4 p^{m + 4} + 2 p^{m + 3} - p^5 - p^4 ) := f_{2}(p,m,n).
\end{align*}
For $n= 2$,  $p\geq 3$ and $m\geq 1$, we have
\begin{align*}
f_{2}(p,m,n) &= \frac{2}{p (p + 1)} (2 p^{2 m + 4} - 8 p^{2 m + 3} + 12 p^{2 m + 2} - 8 p^{2 m + 1} + 2 p^{2 m} - p^{m + 4} - p^{m + 3}\\
&\quad + 7 p^{m + 2} - 7 p^{m + 1} + 2 p^m - p^2 - p).
\end{align*}
Therefore, for $n= 2$, $p\geq 3$ and $m = 1$    we have $f_{2}(p, m, n)= \frac{2}{p + 1} (2 p^5 - 9 p^4 + 11 p^3 - p^2 - 6 p + 1) = \frac{2}{p + 1} (p^3(p - 3)^2 + p^4(p - 3) + 3 p^2(p - 1) + 2 p(p - 3) + 1)> 0$. For $n= 2$, $p\geq 3$ and $m\geq 2$    we have
\begin{align*}
f_{2}(p,m,n) &= \frac{2}{p (p + 1)} \big{(} (2 p^{2 m + 2}(p - 2)^2 - p^{m + 3}) + p^{2 m + 1}(3 p - 8) + p^{m + 2}(p^m - p^2) \\ 
&\quad + (2 p^{2 m} - p^2) + 7 p^{m + 1}(p - 1) + (2 p^m - p)\big{)}> 0.
\end{align*}
Therefore, if $n= 2$,  $p\geq 3$ and $m\geq 1$ then $LE(\mathcal{CCC}(G))  > LE(K_{|V(\mathcal{CCC}(G))|})$ and so  $\mathcal{CCC}(G)$ is   L-hyperenergetic but not L-borderenergetic.

 For $n\geq 3$, $p=2$ and $m\geq 1$   we have
\begin{align*}
f_{2}(p,m,n) &= \frac{1}{12} (2^{2 (m + n)} - 2^{2 m + n + 1} - 5.2^{m + n + 1} + 2^{m + 3} - 24)\\ 
& \geq \frac{1}{12} (4.2^{2m + n + 1} - 2^{2 m + n + 1} - 5.2^{m + n + 1} + 2^{m + 3} - 24)\\
&= \frac{1}{12} \big{(}(2^{m + n + 1}(3.2^m - 5) - 24) +  2^{m + 3}\big{)} > 0.
\end{align*}
Also, for $n\geq 3$, $p\geq 3$ and $m\geq 1$  we have 
\begin{align*}
f_{2}(p,m,n) &= \frac{4}{p^4(p + 1)}( 2p^{ 2m + 2n + 2 }( p - 3 ) + p^{ m + n + 1 }( p^{ m + n } - p^4 ) + p^{ m + n + 1 }(p^{ m + n } - p^3 )\\
&\quad + 2p^{ 2m + 2n }( p - 1 ) + 6p^{2m + n + 2}( p - 1 ) + 2p^{ 2m + n + 1 }( p^n - p^3 ) + p^{ m + n + 1 }( 2p^m - 3p )\\
& \quad + 2p^{ m + 4 }( p - 2 ) + p^4(2p^{ m - 1 } - 1) + p^3 ( 5 p^{ m + n } - p^2 )\big{)} > 0.
\end{align*}
Therefore,   if $n\geq 3$, $p\geq 2$ and $m\geq 1$ then $LE(\mathcal{CCC}(G))  > LE(K_{|V(\mathcal{CCC}(G))|})$ and so  $\mathcal{CCC}(G)$ is   L-hyperenergetic but not L-borderenergetic. Thus, if $n=2$, $p\geq 3$ and $m\geq 1$ or $n\geq 3$, $p\geq 2$ and $m\geq 1$ then  $\mathcal{CCC}(G)$ is L-hyperenergetic but neither hyperenergetic,   borderenergetic, L-borderenergetic, Q-hyperenergetic nor Q-borderenergetic.
\end{proof}
\section{Conclusion}
By Theorem \ref{G(p,n,m)}, it follows that $\spec({\mathcal{T}}), \L-spec({\mathcal{T}})$ and $\Q-spec({\mathcal{T}})$ contain only integers if ${\mathcal{T}} = \mathcal{CCC}(G(p,m,n))$. Therefore,    commuting conjugacy class graph of $G(p,m,n)$ is super integral. The same conclusion has drawn for commuting conjugacy class graph of the groups $D_{2n}, Q_{4m}, SD_{8n}, V_{8n}$ and $U_{(n, m)}$ in \cite{bN20}. In general, it may be interesting to characterize all finite groups whose commuting conjugacy class graphs are super integral. Similar study has been carried in \cite{dN2018} for commuting graphs of finite groups.  

In Theorem \ref{cG(p,m,n)}, various energies of commuting conjugacy class graphs of  $G(p,m,n)$ are compared. It is also observed that  $E(\mathcal{T}) \leq LE(\mathcal{T})$ and $LE^+(\mathcal{T}) \leq LE(\mathcal{T})$. Thus, it follows that E-LE Conjecture of Gutman et al. \cite{bN20,gavbr} holds for ${\mathcal{T}} = \mathcal{CCC}(G(p,m,n))$. However, $\mathcal{CCC}(G(p,m,n))$ provides negative answer to  \cite[Question 1.2]{dBN2020}. Similar conclusion has drawn for commuting conjugacy class graph and commuting   graph of many other families of finite groups in \cite{dN2018} and \cite{dBN2020} respectively.

Finally, in Theorem \ref{hyper-G(p,n,m)}, various energies of $\mathcal{CCC}(G(p,m,n))$ and $K_{|V(G(p,m,n))|}$ are compared and obtained that $\mathcal{CCC}(G(p,m,n))$ is neither hyperenergetic, borderenergetic,   L-borderenergetic   nor Q-borderenergetic. It is also observed that 
$\mathcal{CCC}(G(p,m,n))$ is L-hyperenergetic if   
$n=2$, $p=2$,  $m \geq 3$; $n \geq 3, p \geq 2, m\geq 1$;
and $\mathcal{CCC}(G(p,m,n))$ is Q-hyperenergetic if $n=2$, $p=2$,  $m \geq 4$; $n = 2, p \geq 3, m\geq 1$; $n=3, p = 2, m\geq 2$; or $n\geq 4, p\geq 2, m\geq 1$.

\noindent {\bf Acknowledgements}
    The first author is thankful to Council of Scientific and Industrial Research  for the fellowship (File No. 09/796(0094)/2019-EMR-I).

\end{document}